\documentclass[]{amsart}
\usepackage{amsmath,amsfonts,amssymb,amsthm}
\usepackage{enumerate, color}
 
\DeclareMathOperator{\ord}{ord}

\DeclareMathOperator{\Diff}{Diff}

\DeclareMathOperator{\real}{Re}

\DeclareMathOperator{\Res}{Res}

\DeclareMathOperator{\CS}{CS}

\newcommand\parcial[1]{\dfrac{\partial}{\partial #1}}

\def\N{\mathbb N}
\def\R{\mathbb R}
\def\C{\mathbb C}
\def\Q{\mathbb Q}

\def\P{\mathbb P}

\def\Cd#1{(\C^#1,0)}

\theoremstyle{plain}
\newtheorem{theorem}{Theorem}[section]

\newtheorem{lemma}[theorem]{Lemma}
\newtheorem{Theorem}{Theorem}
\newtheorem{cor}[theorem]{Corollary}
\theoremstyle{definition}

\newtheorem{remark}[theorem]{Remark}
\newtheorem{example}[theorem]{Example}
\newcommand{\obra}[3]{{\sc #1} {\em #2}. {#3}.}
\author{Lorena L\'opez-Hernanz}
\address{Lorena L\'opez-Hernanz\\ Departamento de F\'isica y Matem\'{a}ticas\\
Universidad de Alcal\'a\\ Edificio de Ciencias, Carretera Madrid-Barcelona, Km. 33600\\ 28871 Alcal\'a de Henares, Madrid,
Spain} \email{lorena.lopezh@uah.es}
\author{Rudy Rosas}
\address{Rudy Rosas\\ Departamento de Ciencias, Pontificia Universidad Cat\'olica del Per\'u, Av. Universitaria 1801, Lima, Peru}\email{rudy.rosas@pucp.pe}
\date{}
\thanks{First author partially supported by Ministerio de Econom\'ia y Competitividad, Spain, process MTM2016-77642-C2-1-P; second author, by Vicerrectorado de Investigaci\'on de la Pontificia Universidad Cat\'olica del Per\'u}
\title[Characteristic directions of two-dimensional biholomorphisms]
{Characteristic directions of two-dimensional biholomorphisms}
\begin{document}
\maketitle

\begin{abstract} We prove that for each characteristic direction $[v]$ of a tangent to the identity diffeomorphism of order $k+1$ in $\Cd 2$ there exist either an analytic curve of fixed points tangent to $[v]$ or $k$ parabolic manifolds where all the orbits are tangent to $[v]$, and that at least one of these parabolic manifolds is or contains a parabolic curve. 

\end{abstract}

\section{Introduction}

Let $F\in\Diff\Cd2$ be a tangent to the identity diffeomorphism of order $k+1$, 
$$F(z)=z+F_{k+1}(z)+F_{k+2}(z)+\dots,$$
where $F_j(z)$ is a two-dimensional vector of homogeneous polynomials of degree $j$ and $F_{k+1}(z)\neq 0$. A \emph{characteristic direction} of $F$ is an element $[v]\in \P_{\C}^1$ such that $F_{k+1}(v)=\lambda v$ for some $\lambda\in\C$; it is called \emph{degenerate} if $\lambda=0$, and \emph{non-degenerate} otherwise. Characteristic directions are those complex directions along which some stable dynamics of $F$ can exist: 
\begin{enumerate}[$\circ$]
\item If there exists an irreducible analytic curve pointwise fixed by $F$, then its tangent line 
at $0\in\C^2$ is a characteristic direction of $F$. More generally, assume that there exists an irreducible analytic curve $\mathcal{C}$ invariant by $F$ in the sense of germs, i.e. $\mathcal{C}$ and $F(\mathcal{C})$ define the same germ at $0\in\C^2$. Then the tangent line of $\mathcal{C}$ at $0\in\C^2$ is also a characteristic direction of $F$. %
\item If an orbit $\{F^n(p)\}$ of $F$ is tangent to some complex direction $[v]$ at $0\in\C^2$, then $[v]$ is a characteristic direction of $F$ (see \cite[Proposition 2.3]{Hak}). 
\end{enumerate} 
The following natural question arises: does every characteristic direction $[v]$ of $F$ have some stable dynamics associated to it?  This question has been addressed by several authors. \'Ecalle \cite{Eca} and Hakim \cite{Hak} gave a positive answer in the case where $[v]$ is non-degenerate, showing that there exist at least $k$ parabolic curves tangent to $[v]$. In the case where $[v]$ is degenerate, there are partial answers by Molino \cite{Mol}, Vivas \cite{Viv} or L\'opez et al. \cite{Lop-S, Lop-R-R-S}, among others, which guarantee, under some additional hypotheses, the existence of parabolic curves tangent to $[v]$ or of parabolic domains along $[v]$. In the particular case where $F$ has an isolated fixed point at the origin, even if all of its characteristic directions are degenerate, Abate proved in \cite{Aba} (see also \cite{Bro-C-L}) that one of the characteristic directions of $F$ supports at least $k$ parabolic curves.

In this paper we give a complete positive answer to the question above. Our result is the following:

\begin{Theorem}\label{main1} Let $F\in\Diff\Cd 2$ be a tangent to the identity diffeomorphism of order $k+1$, and let $[v]$ be a characteristic direction of $F$. Then at least one of the following possibilities holds:
\begin{itemize}
\item [1.] There exists an analytic curve pointwise fixed by $F$ and tangent to $[v]$.
\item [2.] There exist at least $k$ invariant sets $\Omega_1, \dots, \Omega_k$, where each $\Omega_i$ is either a parabolic curve tangent to $[v]$ or a parabolic domain along $[v]$ and such that all the orbits in $\Omega_1\cup\dots\cup \Omega_k$ are mutually asymptotic. Moreover, at least one of the invariant sets $\Omega_j$ is a parabolic curve.
\item [3.] There exist at least $k$ parabolic domains $\Omega_1, \dots, \Omega_k$ along $[v]$, where each $\Omega_i$ is foliated by parabolic curves and such that all the orbits in $\Omega_1\cup\dots\cup \Omega_k$ are mutually asymptotic.
\end{itemize}  
In particular, if $F$ has an isolated fixed point then for any characteristic direction $[v]$ there is a parabolic curve tangent to $[v]$.
\end{Theorem}

In \cite{Ast}, Astorg et al. provide examples of polynomial diffeomorphisms in $\C^2$ of the form
$$P(x,y)=\left(x+\frac{\pi^2}4y+x^2+O(x^3), y-y^2+O(y^3)\right)$$
having a wandering Fatou component, and wonder whether these diffeomorphisms have any parabolic curves other than the one contained in $(y=0)$; applying Theorem~\ref{main1}, we will show in Example~\ref{ejemplo} that they do. As another consequence of our results, we will obtain in Corollary~\ref{generalMolino} a generalization of a result by Molino~\cite{Mol}. 

In the last section, we will analyze the particular case of characteristic directions with non-vanishing index, giving a positive answer to a conjecture by Abate in \cite{Aba2}.

\section{Formal vector fields}\label{sec:vectorfields}
In this section we recall some classical results on the resolution of formal vector fields and the existence of separatrices (see for instance \cite{Cam-S2} or \cite{Can-C-D}). Consider a formal vector field in $\Cd2$
$$X=A(x,y)\parcial x+B(x,y)\parcial y,$$ 
where $A,B\in \mathbb{C}[[x,y]]$.  We say that $X$ is \emph{non-singular} if $A$ or $B$ is a unit, otherwise we say that $X$ is \emph{singular}. If $A$ and $B$ have no common factor, we say that $X$ is \emph{saturated}. Any vector field $X$ can be written, in a unique way up to multiplication by a unit, as $$X=f\bar{X},$$ where $f\in\C[[x,y]]$ and $\bar{X}$ is a saturated vector field. 
The formal curve $(f)$ is the \emph{singular locus}
of $X$ and the vector field $\bar{X}$ is the \emph{saturation} of $X$. We say that $X$ is \emph{strictly singular} if its saturation is singular.

An irreducible formal curve $S=(g)$ is a \emph{separatrix} of $X$ if $X(g)\in S$. A separatrix of $X$ is called \emph{strict} if it is also a separatrix of the saturation of $X$; otherwise, it is called \emph{non-strict}.
Any component of the singular locus of $X$ is a separatrix of $X$; these separatrices are called \emph{fixed}. Clearly, a non-strict
separatrix is necessarily fixed.  If $X$ is non-singular at the origin, then the 
formal integral curve through the origin is the only separatrix of $X$. 

Assume that $X$ is singular, let $\pi:(M,D)\to\Cd2$ be the blow-up at the origin and, for any $q\in D=\pi^{-1}(0)$, denote by $X_q$ the transform of $X$ by $\pi$ at $q$, i.e. the unique formal vector field $X_q$ in $(M,q)$ such that $d\pi\cdot X_q=X$. Write $A(x,y)=A_{k+1}(x,y)+A_{k+2}(x,y)+\dots$ and $B(x,y)=B_{k+1}(x,y)+B_{k+2}(x,y)+\dots$ as sum of homogeneous polynomials, where $(A_{k+1},B_{k+1})\neq0$, and put
$$P_X(x,y)=xB_{k+1}(x,y)-yA_{k+1}(x,y).$$
Note that if $q\in D$ is a point where the transform of $X$ is strictly singular, then $q$ corresponds to a zero of $P_X$, and that the tangent line of a separatrix of $X$ is also a zero of $P_X$. If $P_X\equiv 0$, we say that $X$ is \emph{dicritical}, otherwise it is \emph{non-dicritical}. When $X$ is dicritical, the exceptional divisor $D$ is a non-strict separatrix of $X_q$ for all $q\in D$ and $X$ has infinitely many strict separatrices. When $X$ is non-dicritical, $D$ is a strict separatrix of $X_q$ for every point $q\in D$; moreover, each zero of the polynomial $P_X$ either is the tangent line of a fixed separatrix of $X$ or corresponds to a point $q\in D$ such that $X_q$ is strictly singular.

\subsection*{\emph{Reduced vector fields}}
Consider a singular saturated formal vector field $X$ in $\Cd2$. We say that $X$ is \emph{reduced} if the eigenvalues $\lambda_{1},\lambda_{2}$ of its linear part satisfy $\lambda_{1}\neq0$ and $\lambda_{2}/\lambda_{1}\not\in\Q_{>0}$;
if $\lambda_2\neq 0$ we say that $X$ is \emph{non-degenerate}, otherwise $X$ is called a \emph{saddle-node}. Reduced vector fields have exactly two formal separatrices, which are non-singular and transverse. Each of these separatrices is tangent to an eigenspace of the linear part of $X$: it is called \emph{strong} if it is tangent to an eigenspace of a non-zero eigenvalue, otherwise it is called \emph{weak}. Thus, a non-degenerate vector field has two strong separatrices, whereas a saddle-node has a strong one and a weak one. If $X$ is reduced, $\pi:(M,D)\to\Cd2$ is the blow-up at the origin and $q\in D$ corresponds to a separatrix $S$ of $X$, then the saturation of the transform of $X$ at $q$ is reduced: non-degenerate if $S$ is a strong separatrix and a saddle-node if $S$ is a weak separatrix.

\subsection*{\emph{Resolution and Camacho-Sad Theorem}} 
Let $X$ be a singular formal vector field in $\Cd2$. By Seidenberg's Theorem \cite{Sei}, there exists a finite composition $\pi\colon(M,E)\to\Cd2$ of blow-ups at strictly singular points such that for any point $q\in E=\pi^{-1}(0)$ the saturation of the transform of $X$ at $q$ is either non-singular or reduced. Any map $\pi$ as above is called a \emph{resolution} of $X$; as may be expected, there exists a unique minimal such $\pi$, which is called the \emph{minimal resolution} of $X$. A strict separatrix $S$ of $X$ is called \emph{weak} if after a resolution of $X$ the saturation $\bar X_q$ of the transform of $X$ at the point $q$ corresponding to $S$ is reduced and the strict transform of $S$ is a weak separatrix of $\bar X_q$; otherwise, it is called \emph{strong}. It is easy to see that this definition does not depend on the resolution. Note that in the definition of strong separatrices, besides those whose strict transform is a strong separatrix of a reduced saturation of a transform of $X$ after resolution, we also include those whose strict transform after resolution is the integral curve of a non-singular saturation of a transform of $X$; it is easy to see that any of the separatrices of the latter type becomes a separatrix of the former type after an additional blow-up. Camacho-Sad Theorem \cite{Cam-S} guarantees the existence of at least one strict separatrix of $X$, which moreover is strong.

\section{diffeomorphisms}\label{sec:diffeomorphisms}
Let $F\in\Diff (\C^2,0)$ be a tangent to the identity diffeomorphism of order $k+1$, and write
$$F(z)=z+F_{k+1}(z)+F_{k+2}(z)+\dots,$$
where $F_j(z)$ is a vector of homogeneous polynomials of degree $j$ and $F_{k+1}(z)=(p_{k+1}(z),q_{k+1}(z))\neq0$. Clearly, the characteristic directions of $F$ are the zeros of the polynomial $xq_{k+1}(x,y)-yp_{k+1}(x,y)$, writing  $z=(x,y)$. It is well known that there exists a unique formal vector field $X$ in $\Cd2$ of order at least two such that $F=\exp X,$
where 
$$\exp X=\left(\sum_{n=0}^\infty \frac{X^n(x)}{n!},\sum_{n=0}^\infty\frac{X^n(y)}{n!}\right).$$
This vector field is called the \emph{infinitesimal generator} of $F$ and denoted by $\log F$. 

An irreducible formal curve $S=(g)$ is a \emph{separatrix} of $F$ if $g\circ F\in S$; it is \emph{fixed} if there is a convergent generator $h\in S$ such that the set $h=0$ is pointwise fixed by $F$. The following properties relating the separatrices of $F$ and $\log F$ are well known (see for example \cite{Rib} or \cite{Bro-C-L}): a formal curve is a separatrix of $F$ if and only if it is a separatrix of $\log F$, and it is a fixed separatrix of $F$ if and only if it is a fixed separatrix of $\log F$. In particular, the reduced 
singular locus of $\log F$ is convergent and coincides with the set of fixed points of $F$. Moreover, the order of $\log F$ is exactly $k+1$, and its jet of order $k+1$ is
$$p_{k+1}(x,y)\parcial x+q_{k+1}(x,y)\parcial y.$$
Therefore, the polynomial $xq_{k+1}(x,y)-yp_{k+1}(x,y)$, whose zeros are the characteristic directions of $F$, coincides with the polynomial $P_{\log F}(x,y)$ defined in Section~\ref{sec:vectorfields}.

Let $\pi:(M,D)\to\Cd2$ be the blow-up at the origin and let $q\in D$ be a point corresponding to a characteristic direction of $F$. There exists a unique diffeomorphism $F_q\in\Diff(M,q)$, called the \emph{transform} of $F$ by $\pi$ at $q$, such that $\pi\circ F_q=F\circ\pi$. This diffeomorphism $F_q$ is tangent to the identity, its order is greater than or equal to the order of $F$ and its infinitesimal generator is the transform of $\log F$ by $\pi$ at $q$ (see \cite{Bro-C-L}).

An orbit $O$ of $F$ has the \emph{property of iterated tangents} if $O$ converges to the origin and satisfies the following property: if $\pi_1:(M_1,D_1)\to\Cd2$ is the blow-up at the origin, then $\pi_1^{-1}(O)$ converges to a point $p_1\in D_1$; if $\pi_2:(M_2,D_2)\to(M_1,p_1)$ is the blow-up at $p_1$, then $\pi_2^{-1}(\pi_1^{-1}(O))$ converges to a point $p_2\in D_2$, and so on. The sequence $\{p_n\}$ is called the \emph{sequence of iterated tangents} of $O$. Two orbits $O_1$ and $O_2$ of $F$ are \emph{mutually asymptotic} if $O_1$ and $O_2$ have the property of iterated tangents and their sequences of iterated tangents coincide. When an orbit $O$ has the property of iterated tangents and its sequence of iterated tangents coincides with the sequence of infinitely near points of a formal curve $S$, we say that $O$ is \emph{asymptotic} to $S$; in this case, $S$ is necessarily a separatrix of $F$ (see \cite{Lop-R-R-S}).

\subsection*{\emph{Parabolic curves and parabolic domains}}
Let $F\in\Diff \Cd 2$ be a tangent to the identity diffeomorphism. A \emph{parabolic curve} (respectively, a \emph{parabolic domain}) of $F$ is a simply connected complex manifold $\Omega\subset\C^2$ of dimension one (respectively, two) with $0\in\partial \Omega$ which is positively invariant by $F$ and such that $F^n\rightarrow 0$ as $n\rightarrow 0$ uniformly on compact subsets of $\Omega$; if every positive orbit in $\Omega$ is tangent to a direction $[v]\in\P_\C^1$, we say that $\Omega$ is a parabolic curve \emph{tangent} to $[v]$ (respectively, a parabolic domain \emph{along} $[v]$).
 
The following theorem, proved by L\'opez, Raissy, Rib\'on and Sanz in \cite{Lop-R-R-S}, relates the existence of separatrices with the existence of parabolic curves. 

\begin{theorem}[{\cite[Theorem 7.1]{Lop-R-R-S}}]\label{lrrs} Let $F\in\Diff\Cd2$ be a tangent to the identity diffeomorphism of order $k+1$
and let $S$ be a separatrix of $F$. Then, either 
$S$ is fixed or there exist at least $k$ invariant sets $\Omega_1,...,\Omega_k$, where each $\Omega_i$ is either a parabolic curve or a parabolic domain of $F$ and such that every orbit in $\Omega_1\cup\dots\cup \Omega_k$ is asymptotic to $S$. Moreover, at least one of the sets $\Omega_j$ is a parabolic curve.
\end{theorem}

\section{Proof of Theorem~\ref{main1}}\label{sec:proof}

Let $F\in\Diff\Cd 2$ be tangent to the identity and let $[v]$ be a characteristic direction of $F$. Let $\pi\colon (M,D)\to(\mathbb{C}^2,0)$ be the blow-up at $0\in\mathbb{C}^2$, and let $q\in D$ be the point corresponding to $[v]$. Denote by $X$ the infinitesimal generator of $F$ and let $X_q$ be the transform of $X$ by $\pi$ at $q$. If $X_q$ has a separatrix $S$ different from $D$, then $S$ defines a separatrix of $X$ and Theorem \ref{lrrs} finishes the proof. Otherwise, $D$ is the only separatrix of $X_q$, and by Camacho-Sad Theorem it is a strict separatrix of $X_q$, so $X$ is non-dicritical. Hence, since $[v]$ is a characteristic direction of $F$, either there is a fixed separatrix tangent to $[v]$, and we are done, or $X_q$ is strictly singular. Thus Theorem \ref{main1} is an immediate consequence of the following result. 

\begin{theorem}\label{parabolic-domain} Let $F\in\Diff\Cd 2$ be tangent to the identity of order $k+1$ and such that $\log F$ is strictly singular. Assume that $\log F$ has exactly one strict separatrix $S$, which moreover is non-singular. Then there exist at least $k$ parabolic domains $\Omega_1,\dots,\Omega_k$ with the following properties:
\begin{enumerate}[1.]
\item For each $j$ there exists an injective holomorphic map $\phi_j\colon \Omega_j\to\mathbb{C}^2$ such that  $$\phi_j\circ F\circ \phi_j^{-1}(x,y)=(x+1,y).$$  
\item For all $\mu\in\C$ and all $j$, the set $\phi_j(\Omega_j)\cap\{y=\mu\}$ is simply connected. Therefore, $\Omega_j$ is foliated by parabolic curves defined by the sets $y=\mu$, for $\mu\in\C$. 
\item After a finite number of blow-ups all the orbits in $\Omega_1\cup\dots\cup\Omega_k$ converge to the same point in the exceptional divisor $E$ and are asymptotic to the same component of $E$.
\end{enumerate}
\end{theorem}

The proof of Theorem~\ref{parabolic-domain} is based on Lemma~\ref{unalisa} below. We say that $X$ \emph{contains a saddle-node} if for some resolution $\pi$ of $X$ there exists a point in the exceptional divisor at which the saturation of the transform of $X$ is a saddle-node; it is easy to see that this property does not depend on the choice of the resolution. 
\begin{lemma}\label{unalisa}  Let $X$ be a singular saturated formal vector field in $\Cd2$. If $X$ has exactly one separatrix $S$, which moreover is 
non-singular, then $X$ contains a saddle-node. 
\end{lemma} 

In the case of an analytic vector field, Lemma~\ref{unalisa} is a direct consequence of a result by Camacho, Lins Neto and Sad in \cite[Theorem 2]{Cam-L-S}, and their proof also works in the formal context. Alternatively, we can deduce the formal statement from the analytic one considering an appropriate truncation of the vector field. Anyhow, we will provide an alternative proof of the lemma by induction, and for technical reasons we will include the following result, which also appears in \cite[p. 162]{Cam-L-S}. 

\begin{lemma}\label{doslisas}  Let $X$ be a saturated formal vector field in $\Cd2$. Suppose that $X$ has exactly two separatrices $S_1$ and $S_2$, which moreover are non-singular and transverse. Then, either $X$ is reduced or $X$ contains a saddle-node. 
\end{lemma}

\begin{proof}[Proof of Lemmas~\ref{unalisa} and \ref{doslisas}]
Both lemmas are clearly true if $X$ is reduced.
Suppose that they hold for vector fields whose minimal resolution is obtained with at  
most $n\ge 0$ blow-ups. Let $X$ be as in Lemma~\ref{unalisa} and suppose that the minimal resolution of $X$ is obtained after $n+1$ blow-ups. Let $\pi$ be the blow-up at the origin and let $D$ be the exceptional divisor. Clearly $X$ is non-dicritical, otherwise $X$ would have infinitely many separatrices. Let $p$ be the point in $D$ corresponding to the tangent of $S$. Since the transform $X_p$ of $X$ at $p$ has two strict separatrices, it is strictly singular. Suppose first that the transform $X_q$ is strictly singular for some other point $q\in D\setminus\{p\}$. Note that $D$ is the only separatrix of the saturation $\bar X_q$ of $X_q$, otherwise $X$ would have a separatrix different from $S$. Therefore, by the inductive hypothesis Lemma~\ref{unalisa} holds for $\bar{X}_q$, so $\bar{X}_q$ contains a saddle-node and therefore $X$ contains a
saddle-node. Assume then that $p$ is the only point in $D$ such that $X_p$ is strictly singular. Since $S$ is non-singular, its strict transform by $\pi$ is transverse to $D$, so by the inductive hypothesis Lemma~\ref{doslisas} holds for the saturation $\bar{X}_p$ of $X_p$. If $\bar{X}_p$ contains a saddle-node, then $X$ contains a saddle-node and we are done. We assume then that $\bar{X}_p$ is reduced. 
Take coordinates $(x,t)$ at $p$ such that $\pi(x,t)=(x,tx)$. If 
$$A_{k+1}(x,y)\parcial x+B_{k+1}(x,y)\parcial y$$ 
is the first nonzero jet of $X$, then 
$$\bar{X}_{p}=O(x)\parcial x+\left[P_X(1,t)+O(x)\right]\parcial t,$$
where $P_X(x,y)=xB_{k+1}(x,y)-yA_{k+1}(x,y)$. Since $P_X$ has order at least two and $p$ is the only point in $D$ such that $X_p$ is strictly singular, we have that $P_X(1,t)$ has order at least two at $t=0$. Then $$\bar{X}_{p}=O(x)\parcial x+\left[O(t^2)+O(x)\right]\parcial t$$
and its linear part is of the form
$$O(x)\parcial x+O(x)\parcial t.$$ 
Therefore, since $\bar{X}_p$ is reduced, it is necessarily a saddle-node and we are done.
Consider now a vector field $X$ as in Lemma~\ref{doslisas} whose resolution is obtained after $n+1$ blow-ups, and let $\pi:(M,D)\to\Cd2$ be the blow-up at the origin. Clearly $X$ is again non-dicritical, the strict transforms of $S_1$ and $S_2$ are transverse to $D$ at points $p_1$ and $p_2$, respectively, and the transforms $X_{p_1}$ and $X_{p_2}$ of $X$ are strictly singular. Suppose that the transform $X_q$ is strictly singular at some point $q\in D\setminus\{p_1,p_2\}$. Observe that $D$ is the only separatrix of the saturation $\bar{X}_q$ of $X_q$, otherwise $X$ would have a separatrix different from $S_1$ and $S_2$. Then, by the inductive hypothesis, Lemma~\ref{unalisa} holds for $\bar{X}_q$ and therefore $\bar{X}_q$ contains a saddle-node. Thus, we assume that $p_1$ and $p_2$ are the only points in $D$ such that $X_{p_1}$ and $X_{p_2}$ are strictly singular. Again by the inductive hypothesis, Lemma~\ref{doslisas} holds for their saturations $\bar{X}_{p_1}$ and $\bar{X}_{p_2}$. If $\bar{X}_{p_1}$ or $\bar{X}_{p_2}$ is not reduced, $X$ contains a saddle-node. We assume then that $\bar{X}_{p_1}$ and $\bar{X}_{p_2}$ are reduced. If $\bar X_{p_1}$ or $\bar X_{p_2}$ is a saddle-node, $X$ contains a saddle-node. Otherwise, $\bar{X}_{p_1}$ and $\bar{X}_{p_2}$ are non-degenerate. In this case, taking formal coordinates so that $S_1$ and $S_2$ are the coordinate axes, a simple computation on the expression of $X_{p_1}$ and $X_{p_2}$ shows that the order of $X$ is one and that $X$ is reduced.
\end{proof}

\begin{proof}[Proof of Theorem~\ref{parabolic-domain}] Write $X=\log F$ and let $\pi\colon (M,E)\to\Cd2$ be 
a resolution of $X$. By Lemma~\ref{unalisa}, there exists a point $q\in E=\pi^{-1}(0)$ such that the saturation $\bar{X}_q$ of the transform $X_q$ of $X$ at $q$ is a saddle-node. 
Observe that, since the order of $X$ is $k+1\ge 2$, the components of $E$ at $q$ are contained in the singular locus of $X_q$, each of them with multiplicity at least $k$. Up to some additional blow-ups, we can assume that they are the only fixed separatrices of $X_q$. 
Suppose that the weak separatrix $S_w$ of $\bar{X}_q$ is not contained in the exceptional divisor. Then 
$S_w$ defines a strict separatrix of $X$, which coincides with $S$ since $S$ is the only strict separatrix of $X$. This is impossible, since
by Camacho-Sad Theorem the only strict separatrix of $X$ cannot be a weak separatrix. Therefore $S_w$ is contained in $E$ and so it is contained in
the singular locus of $X_q$. Consider holomorphic coordinates $(z,u)$ at $q$ such that $S_w=\{z=0\}$. Clearly, one component of $E$ is then given by $\{z=0\}$; if $q$ belongs to two components of $E$, we can assume that the other one is given by $\{u=0\}$.
Then, the vector field $X_q$ can be written in the form
\begin{equation}\label{eq:campoLiz}
X_q=z^ru^m\left[z\left(a+G(z,u)\right)\parcial z+\left(bz+H(z,u)\right)\parcial u\right],
\end{equation} 
where $r\ge k\ge1$, $m\ge 0$, $a\neq0$, $b\in\mathbb{C}$ and $G,H\in\mathbb{C}[[z,u]]$ satisfy $\ord G\ge 1$, $\ord H\ge 2$ and $H\not\in z\C[[z,u]]$. Hence, the transform $F_q$ of $F$ at $q$, which is equal to $\exp X_q$, has the form 
$$F_q=\left(z+z^{r+1}u^m\left[a+O(z,u)\right], u+z^ru^m\left[cu^{p+1}+O(z,u^{p+2})\right]\right),$$ 
with $c\neq0$ and $p\ge 1$. This kind of diffeomorphisms are studied by Vivas in
\cite[p.~2032]{Viv} and she proves that, after a linear change of coordinates so that $a=c=-1$, there exists a domain $\widetilde{\Omega}$ of the form
$$\widetilde\Omega=\left\{(z,u)\in\C^2: |z^ru^m-\varepsilon|<\varepsilon, |\arg(z^ru^m)|<\eta, |u^p-\delta|<\delta, |z|<|u|^M\right\},$$
for some $\varepsilon, \delta<1/2$ and $\eta>0$ sufficiently small and some $M\ge2$ sufficiently large, whose connected components $\widetilde\Omega_1, ...,\widetilde\Omega_{rp}$ are parabolic domains of $F_q$ and which moreover satisfies the following properties:
\begin{enumerate}[1.]
\item for any $j=1,...,rp$ there exists an injective holomorphic map $\varphi_j\colon \widetilde{\Omega}_j\to\C^2$ such that $\varphi_j\circ F_q\circ\varphi_j^{-1}(x,y)=(x+1,y);$
\item for any $\mu\in\C$ and any $j=1,...,rp$, the set $\varphi_j(\widetilde{\Omega}_j)\cap\{y=\mu\}$ is simply connected. 
\end{enumerate}
Clearly, the images by $\pi$ of the sets $\widetilde\Omega_1, ...,\widetilde\Omega_{rp}$ are parabolic domains of $F$ satisfying the two first properties of Theorem~\ref{parabolic-domain}. To prove the third one, we will show that any orbit in $\widetilde\Omega$ is asymptotic to $z=0$. Consider a point $(z,u)\in\widetilde\Omega$, and denote its orbit by $\{(z_n,u_n)\}$. Observe first that, since $|z|<|u|^M$ and $M\ge2$, we have
\begin{align*}
z_1^ru_1^m&=z^ru^m\left[1-rz^ru^m+z^ru^mO\left(u,zu^{-1}\right)\right]\\
&=z^ru^m-r\left(z^ru^m\right)^2+\left(z^ru^m\right)^2O(u),
\end{align*}
so arguing as in the classical Leau-Fatou Flower Theorem \cite{Lea, Fat} we obtain that $\{z_n^ru_n^m\}$ converges tangentially to the direction $\R^+$. Consider an integer $N\in\N$. We have that
\begin{align*}
\frac{|z_{n+1}|}{|u_{n+1}|^N}&=\frac{|z_n|}{|u_n|^N}\left|1-z_n^ru_n^m+z_n^ru_n^mO\left(u_n,z_nu_n^{-1}\right)\right|\\
&=\frac{|z_n|}{|u_n|^N}\left|1-z_n^ru_n^m+z_n^ru_n^mO(u_n)\right|.
\end{align*}
Since $\real(z_n^ru_n^m)>0$, $u_n\to0$ and $\{z_n^ru_n^m\}$ converges to $0$ with $\R^+$ as tangent direction, we can assume that $|z_n^ru_n^m|$, $\arg(z_n^ru_n^m)$ and $|u_n|$ are small enough so that the term $\left|1-z_n^ru_n^m+z_n^ru_n^mO(u_n)\right|$ above is bounded by 1 for all $n\ge n_0=n_0(N)$. This shows that $|z_n|\le C|u_n|^N$ for all $n\ge n_0$, with $C=|z_{n_0}|/|u_{n_0}|^N$, so the orbit is asymptotic to $z=0$. This ends the proof of Theorem~\ref{parabolic-domain}.
\end{proof}

\begin{remark}
It is worth to notice than in case 2 of Theorem~\ref{main1} all the orbits in $\Omega_1\cup\dots\cup\Omega_k$ are asymptotic to a separatrix of $F$, whether in case 3 all the orbits in $\Omega_1\cup\dots\cup\Omega_k$ have the property of iterated tangents, but are not asymptotic to any formal curve. Contrary to what we have in case 3, at the moment we cannot guarantee that the parabolic domains appearing in case 2 are foliated by parabolic curves.
\end{remark}

As a corollary of Theorems~\ref{lrrs} and \ref{parabolic-domain} we obtain the following result, which generalizes Molino's main theorem in~\cite[Theorem 1.6]{Mol}.
\begin{cor}\label{generalMolino}
Let $E$ be a smooth Riemann surface in a complex surface $M$, and let $F$ be a tangential germ of diffeomorphism of $M$ fixing $E$ pointwise and whose order of contact with $E$ is $k+1$. Let $p\in E$ be a singular point which is not a corner. Then there exists a parabolic curve of $F$ at $p$. More precisely, at least one of the following possibilities holds:
\begin{itemize}
\item[1.] There exist $k$ invariant sets $\Omega_1, \dots, \Omega_k$  of $F$ at $p$, where each $\Omega_i$ is either a parabolic curve or a parabolic domain and such that all the orbits in $\Omega_1\cup\dots\cup \Omega_k$ are mutually asymptotic. Moreover, at least one of the invariant sets $\Omega_j$ is a parabolic curve.
\item[2.] There exist $k$ parabolic domains $\Omega_1, \dots, \Omega_k$ of $F$ at $p$, where each $\Omega_i$ is foliated by parabolic curves and such that all the orbits in $\Omega_1\cup\dots\cup \Omega_k$ are mutually asymptotic.
\end{itemize}
\end{cor}
\begin{proof}
Let $F_p$ be the germ of $F$ at $p$, which is a tangent to the identity diffeomorphism of order $k+1$. Since $E$ is smooth, $E$ defines a non-singular germ of curve at $p$. The fact of $F$ being tangential means that $E$ is a strict separatrix of $\log F_p$. Being a singular point means that $\log F_p$ is strictly singular, and being not a corner means that there is no other fixed separatrix of $\log F_p$. Therefore the result follows immediately from Theorem~\ref{lrrs} if there is another strict separatrix of $\log F_p$ and from Theorem~\ref{parabolic-domain} otherwise. 
\end{proof}

\begin{example}\label{ejemplo}
In \cite{Ast}, Astorg et al. show the existence of polynomial diffeomorphisms in $\C^2$ possessing a wandering Fatou component. These diffeomorphisms have the form
$$P(x,y)=\left(x+\frac{\pi^2}4y+x^2+O(x^3), y-y^2+O(y^3)\right),$$
and the authors wonder in \cite[p. 275]{Ast} if there are any parabolic curves apart from the one contained in the separatrix $S=(y=0)$. If $\pi$ is the blow-up at the origin and $q\in\pi^{-1}(0)$ is the point corresponding to the direction $[1:0]$, which is the only invariant line of $DP(0)$, the transform of $P$ by $\pi$ at $q$ has the form
$$P_q(x,t)=\left(x+x^2+\frac{\pi^2}4 xt+O(x^3),t-xt-\frac{\pi^2}4t^2+O(x^2t,xt^2,t^3)\right),$$
so it has three characteristic directions: $[1:0]$, $[0:1]$ and $[-\pi^2/4:1]$.
The two first ones are non-degenerate and give rise to two parabolic curves: the first one is contained in $(t=0)$, so it corresponds to the parabolic curve contained in $S$ after blow-down, and the second one is contained in the exceptional divisor $\pi^{-1}(0)$, so it disappears after blow-down. Since $P$ has an isolated fixed point at the origin, $P_q$ has no curve of fixed points tangent to $[-\pi^2/4:1]$. Then, by Theorem~\ref{main1}, there is at least one parabolic curve of $P_q$ tangent to that direction, which defines after blow-down a parabolic curve of $P$ not contained in $S$.
\end{example}

\section{Characteristic directions with non-vanishing index}\label{sec:final}

In this section we will focus on characteristic directions with non-vanishing index, which were studied under an additional asumption by Molino in \cite{Mol}. We begin by recalling the definition of the Camacho-Sad index of a vector field relative to a separatrix. Although the index can be defined for an arbitrary separatrix (see for example \cite{Bru}) we will only be interested in the case of non-singular ones. Consider a formal vector field $X$ in $\Cd2$ and let $S$ be a non-singular strict separatrix of $X$. In appropriate formal coordinates $(x,y)$ we have that $S=(y)$, so the saturation $\bar X$ of $X$ can be written as 
$$\bar X=A(x,y)\parcial x+yB(x,y)\parcial y,$$ 
with $A,B\in\C[[x,y]]$. The \emph{Camacho-Sad index} of $X$ relative to $S$ is defined as
$$\CS(X,S)=\Res_0\left(\frac{B(x,0)}{A(x,0)}\right).$$

Let $F\in\Diff\Cd2$ be a tangent to the identity diffeomorphism. Let $S$ be a non-singular separatrix of $F$, and assume that it is a strict separatrix of $\log F$. The \emph{residual index} of $F$ along $S$, introduced by Abate in \cite{Aba}, can defined as
$$\iota(F,S)=\CS(\log F,S).$$

We will prove the following result, that gives a positive answer to a conjecture by Abate in \cite[Conjecture 3.9]{Aba2}.

\begin{theorem}\label{generalAbate}
Let $F\in\Diff\Cd2$ be a tangent to the identity diffeomorphism of order $k+1$, and let $[v]$ be a characteristic direction of $F$. Let $\pi$ be the blow-up at the origin, $p\in D=\pi^{-1}(0)$ be the point corresponding to $[v]$ and $F_p$ be the transform of $F$ at $p$. If $\log F$ is non-dicritical and 
$\iota(F_p,D)\neq0$, then at least one of the following possibilities holds:
\begin{itemize}
\item [1.] There exists an analytic curve pointwise fixed by $F$ and tangent to $[v]$.
\item [2.] There exist at least $k$ parabolic curves tangent to $[v]$ where all the orbits are asymptotic to a strong separatrix of $\log F$.
\item [3.] There exist at least $k$ parabolic domains along $[v]$ which are foliated by parabolic curves and where all the orbits are mutually asymptotic.
\end{itemize}  
In particular, if $F$ has an isolated fixed point then there exist at least $k$ parabolic curves tangent to $[v]$.
\end{theorem}

The proof of Theorem~\ref{generalAbate} is based on Lemma~\ref{unafuerte-dosfuertes}. Consider a singular formal vector field $X$, ant let $\pi$ be a resolution of $X$. We say, following the terminology of \cite{Mat-S}, that $X$ is a \emph{second type} vector field if none of the saddle-nodes appearing in the resolution $\pi$ have their weak separatrices contained in the exceptional divisor; it is easy to see that this definition does not depend on the choice of the resolution.

\begin{lemma}\label{unafuerte-dosfuertes}
Let $X$ be a second type formal vector field in $\Cd2$. 
\begin{enumerate}[a)]
\item If $X$ has exactly one strong separatrix $S$, which moreover is non-singular, then $\CS(X,S)=0.$
\item If $X$ has exactly two strong separatrices $S_1$ and $S_2$, which moreover are non-singular and transverse, then
$\CS(X,S_1)\CS(X,S_2)=1.$
\end{enumerate}
\end{lemma}

Before proving Lemma~\ref{unafuerte-dosfuertes}, we recall some of the properties of the Camacho-Sad index (see \cite{Cam-S}):
\begin{enumerate}[1.]
\item If $X$ is not strictly singular and $S$ is the formal integral curve through 0 of the saturation of $X$, then clearly $\CS(X,S)=0$.
\item If $\pi$ is the blow-up at the origin, $\tilde S$ is the strict transform of $S$, $p=\tilde S\cap\pi^{-1}(0)$ and $X_p$ is the transform of $X$ at $p$, then 
$$\CS(X_p,\tilde S)=\CS(X,S)-1.$$
\item If $X$ is non-dicritical, $\pi$ is the blow-up at the origin and $D$ is the exceptional divisor, then
$$\sum_{q\in D}\CS(X_q,D)=-1,$$
where $X_q$ is the transform of $X$ at $q$.
\item Suppose that the saturation $\bar X$ of $X$ is reduced, and let $S_1$ and $S_2$ be its separatrices. 
\begin{enumerate}[-]
\item If $\bar X$ is non-degenerate, then $\CS(X,S_1)\CS(X,S_2)=1$. 
\item If $\bar X$ is a saddle-node and $S_1$ is the strong separatrix, then $\CS(X,S_1)=0$.
\end{enumerate}
\end{enumerate}

\begin{proof}[Proof of Lemma~\ref{unafuerte-dosfuertes}] Observe first that any vector field satisfying the hypothesis of a) or b) is non-dicritical, otherwise it would have infinitely many strong separatrices. By properties 1 and 4 above, the lemma is true if the saturation of $X$ is non-singular or reduced. Suppose that it holds for vector fields whose minimal resolution is obtained with at most $n\ge 0$ blow-ups. Let $X$ be as in a) and suppose that its minimal resolution is obtained after $n+1$ blow-ups. Let $\pi$ be the blow-up at the origin and let $D$ be the exceptional divisor. Let $\tilde S$ be the strict transform of $S$ and $p=\tilde S\cap D$. Observe that $D$ is a strict separatrix of the transform $X_p$ of $X$ at $p$, since $X$ is non-dicritical, and that it is a strong separatrix of $X_p$, since $X$ is a second type vector field. Let $q\in D\setminus\{p\}$ be any point such that the transform $X_q$ of $X$ at $q$ is strictly singular. Since $X$ is non-dicritical, $D$ is a strict separatrix of $X_q$. Moreover, none of the other separatrices of $X_q$ are strong, otherwise $X$ would have a strong separatrix different from $S$. Then, by Camacho-Sad Theorem $D$ is a strong separatrix of $X_q$. Therefore, by the inductive hypothesis, the lemma holds for $X_q$, so $\CS(X_q,D)=0.$ Then, by property 3 above we deduce that $\CS(X_p,D)=-1$. Note that $D$ and the strict transform $\tilde S$ of $S$ are the only strong separatrices of $X_p$, otherwise $X$ would have a strong separatrix different from $S$. Then by the inductive hypothesis the lemma holds for $X_p$, so we obtain that $\CS(X_p,\tilde S)=-1$ and therefore $\CS(X,S)=0$, by property 2. Consider now a vector field $X$ as in b) whose minimal resolution is obtained after $n+1$ blow-ups, and let $\pi$ be the blow-up at the origin.
Let $p_1$ and $p_2$ be the points in the exceptional divisor $D$ corresponding to the tangents of $S_1$ and $S_2$, respectively. Let $q$ be any point in $D\setminus\{p_1,p_2\}$ such that the transform $X_q$ of $X$ at $q$ is strictly singular. As in the previous case, $D$ is a strict separatrix of $X_q$ and none of the other separatrices of $X_q$ are strong, so $D$ is a strong separatrix of $X_q$. Then, by the inductive hypothesis the lemma holds for $X_q$ and therefore $\CS(X_q,D)=0.$ 
By property 3, we obtain 
$$\CS(X_{p_1},D)+\CS(X_{p_2},D)=-1,$$
where $X_{p_1}$ and $X_{p_2}$ are the transforms of $X$ at $p_1$ and $p_2$. Observe that, since $X$ is a second type vector field, $D$ is necessarily a strong separatrix of $X_{p_1}$ and $X_{p_2}$. Then, by the inductive hypothesis the lemma holds for $X_{p_1}$ and $X_{p_2}$ and therefore we have 
$$\CS(X_{p_1},\tilde S_1)\CS(X_{p_1},D)=\CS(X_{p_2},\tilde S_2)\CS(X_{p_2},D)=1,$$
where $\tilde S_1$ and $\tilde S_2$ are the strict transforms of $S_1$ and $S_2$, respectively. From the two last equations we easily obtain, using property 2, that 
$\CS(X,S_1)\CS(X,S_2)=1.$\end{proof}

\begin{proof}[Proof of Theorem~\ref{generalAbate}]
Denote $X_p=\log F_p$ and let $\Pi:(M,E)\to\Cd 2$ be a resolution of $X_p$. Suppose first that $X_p$ has a strong separatrix $S$ different from $D$. If $S$ is fixed, we are done. Otherwise, up to an additional blow-up in case the transform of $X$ at the point corresponding to $S$ is not strictly singular, there exists a point $q$ belonging to exactly one component of $E$ such that the saturation of the transform $X_q$ of $X_p$ at $q$ is reduced and its strong separatrix, which is the strict transform of $S$, is transverse to $E$. Up to some additional blow-ups so that $E$ is the only fixed separatrix of $X_q$, there are some local coordinates $(x,y)$ at $q$ such that $E=\{x=0\}$ and such that $X_q$ is written as
$$X_q=x^r\left[\left(\lambda_1x+G(x,y)\right)\parcial x+\left(\lambda_2 y+H(x,y)\right)\parcial y\right]$$
with $r\ge k$, $\lambda_1\neq0$, $\lambda_2\in\C$ and $G,H\in\C[[x,y]]$, with $\lambda_2/\lambda_1\not\in\Q_{\ge0}$ and $\ord G, \ord H\ge2$.
Then, $[1,0]$ is a non-degenerate characteristic direction of the transform $F_q=\exp X_q$ of $F_p$ at $q$ and the existence of $k$ parabolic curves follows from \'Ecalle and Hakim Theorem \cite{Eca,Hak}. Assume now that $X_p$ has no strong separatrices different from $D$. Then, by Camacho-Sad Theorem $D$ is a strong separatrix of $X_p$, so $X_p$ cannot be a second type vector field, otherwise $\CS(X_p,D)=\iota(F_p,D)=0$ by Lemma~\ref{unafuerte-dosfuertes}. Therefore, there exists a point $q\in E$ such that the saturation of the transform $X_q$ of $X_p$ at $q$ is a saddle-node whose weak separatrix is contained in $E$, and arguing exactly as in the proof of Theorem~\ref{parabolic-domain} we can find coordinates $(z,u)$ at $q$ such that
$X_q$ is written as in \eqref{eq:campoLiz} and the existence of $k$ parabolic domains follows from Vivas' results in \cite{Viv}.
\end{proof}


\begin{thebibliography}{99}
\bibitem{Aba} \obra{Abate, M.}{The residual index and the dynamics of holomorphic maps tangent to the identity}{Duke Math. J. 107 (2001), no. 1, 173--207}

\bibitem{Aba2} \obra{Abate, M.}{Fatou flowers and parabolic curves}{Complex analysis and geometry, 1--39, Springer Proc. Math. Stat., 144, Springer, Tokyo, 2015} 

\bibitem{Ast} \obra{Astorg, M.; Buff, X.; Dujardin, R.; Peters, H.; Raissy, J.}{A two-dimensional polynomial mapping with a wandering Fatou component}
{Ann. of Math. (2) 184 (2016), no. 1, 263--313}

\bibitem{Bro-C-L} \obra{Brochero, F.E.; Cano, F.; L\'opez-Hernanz, L.}{Parabolic curves for diffeomorphisms in $\C^2$}{Publ. Mat. 52 (2008), no. 1, 189--194}

\bibitem{Bru} \obra{Brunella, M.}{Some remarks on indices of holomorphic vector fields}{Publ. Mat. 41 (1997), no. 2, 527--544}

\bibitem{Cam-S} \obra{Camacho, C; Sad, P.}{Invariant varieties through singularities of holomorphic vector fields}{Ann. of Math. (2)  115  (1982), no. 3, 579--595}

\bibitem{Cam-S2}\obra{Camacho, C; Sad, P.}{Pontos singulares de equa\c{c}\~oes diferenciais anal\'iticas}{16$^{\rm o}$ Col\'oquio Brasileiro de Matem\'atica, Instituto de Matem\'atica Pura e Aplicada (IMPA), Rio de Janeiro, 1987}

\bibitem{Cam-L-S} \obra{Camacho, C; Lins Neto, A.; Sad, P.}{Topological invariants and equidesingularization for holomorphic vector fields}
{J. Differential Geom. 20 (1984), no. 1, 143--174}

\bibitem{Can-C-D} \obra{Cano, F.; Cerveau, D.; D\'{e}serti, J.}{Th\'{e}orie \'{e}l\'{e}mentaire des feuilletages holomorphes singuliers}
{\'{E}ditions Belin, 2013. Collection ``\'{E}chelles''}
 
\bibitem{Eca} \obra{\'Ecalle, J.}{Les fonctions r\'{e}surgentes. Tome III. L'\'{e}quation du pont et la classification
analytique des objects locaux} {Publications Math\'{e}matiques
d'Orsay, 85-5. Universit\'{e} de Paris-Sud, D\'{e}partement de
Math\'{e}matiques, Orsay, 1985}

\bibitem{Fat} \obra{Fatou, P.}{Sur les \'{e}quations fonctionelles}{Bull. Soc. Math. France 47 (1919) 161--271}

\bibitem{Hak} \obra{Hakim, M.}{Analytic transformations of $(\C^p,0)$ tangent to the identity}{Duke Math. J. 92 (1998), no. 2, 403--428}

\bibitem{Lea} \obra{Leau, L.}{\'{E}tude sur les \'{e}quations fonctionelles \`{a} une ou \`{a} plusieurs variables}{Ann. Fac. Sci. Toulouse Sci. Math. Sci. Phys. 11 (1897) 1--110}

\bibitem{Lop-S} \obra{L\'opez-Hernanz, L; Sanz S\'anchez, F.}{Parabolic curves of diffeomorphisms asymptotic to formal invariant curves}{J. f\"{u}r die reine und ang. Math. doi:10.1515/crelle-2015-0064}

\bibitem{Lop-R-R-S} \obra{L\'opez-Hernanz, L; Raissy, J; Rib\'on, J; Sanz-S\'anchez, F.}{Stable manifolds of two-dimensional biholomorphisms asymptotic to formal curves}{Preprint arXiv:1710.03728, 2017}

\bibitem{Mat-S} \obra{Mattei, J.F.; Salem, E.}{Modules formels locaux de feuilletages holomorphes}{Preprint arXiv:math/0402256, 2004}

\bibitem{Mol}\obra{Molino, L.}{The dynamics of maps tangent to the identity and with nonvanishing index}{Trans. Amer. Math. Soc. 361 (2009), no. 3, 1597--1623}

\bibitem{Rib} \obra{Rib\'on, J.}{Families of diffeomorphisms without periodic curves}{Michigan Math. J. 53 (2005), no. 2, 243--256}

\bibitem{Sei}\obra{Seidenberg, A.}{Reduction of the singularities of the differential equation $Ady=Bdx$}{Am. J. of Math., 90 (1968), 248--269}

\bibitem{Viv}\obra{Vivas, L.}{Degenerate characteristic directions for maps tangent to the identity}{Indiana Univ. Math. J. 61 (2012), no. 6, 2019--2040}
\end{thebibliography}
\end{document}